\documentclass[11pt]{amsart}

\usepackage{amsmath}              
\usepackage{amsfonts}              
\usepackage{amsthm} 
\usepackage[pdftex]{graphicx}
\usepackage{mdwlist}
\usepackage{amssymb}
\usepackage{multicol}
\usepackage{pifont}

\newtheorem{thm}{Theorem}[section]
\newtheorem{lem}[thm]{Lemma}
\newtheorem{prop}[thm]{Proposition}
\newtheorem{cor}[thm]{Corollary}

\newtheorem{defn}[thm]{Definition}

\newtheorem*{theoremb}{Theorem}

\allowdisplaybreaks[3]  %allow page breaks within large equations.

\begin{document}

\title[The Generic Initial System of a Complete Intersection]{The Limiting Polytope of the Generic Initial System of a Complete Intersection}
\author{Sarah Mayes}

\maketitle
\vspace*{-2em} %%reduce space between title and abstract

\begin{abstract}

Consider a complete intersection $I$ of type $(d_1, \dots, d_r)$ in a polynomial ring over a field of characteristic 0.  We study the graded system of ideals $\{ \text{gin}(I^n) \}_n$ obtained by taking the reverse lexicographic generic initial ideals of the powers of $I$ and describe its asymptotic behavior.  This behavior is nicely captured by the \textit{limiting polytope} which is shown to depend only on the type of the complete intersection.
\end{abstract}
		\let\thefootnote\relax\footnotetext{Partially supported by the National Sciences and Engineering Research Council of Canada.}
\section{Introduction}

The asymptotic behavior of algebraic objects has been a fruitful research trend of the past twenty years, motivated by the philosophy that there is often a uniformity achieved in the limit that is hidden when studying individual objects.  Significant work along these lines includes: Huneke's uniform Artin-Rees lemma \cite{Huneke92}; Siu's work on the Fujita conjecture \cite{Siu01}; Ein, Lazarsfeld, and Smith's introduction of graded systems and asymptotic multiplier ideals \cite{ELS01}; and, most recently, Eisenbud and Schreyer's proof of the Boij-S\"{o}derberg conjectures \cite{ES09}.  We study the asymptotic behavior of generic initial ideals - a rich research topic in their own right - using techniques similar to those in \cite{ELS01}.  

Consider a homogeneous ideal $I$ in a polynomial ring $R=K[x_1, \dots, x_m]$ with the standard grading and some fixed term order.  The generic initial ideal of $I$, $\text{gin}(I)$, is a coordinate-independent version of the initial ideal; as a monomial ideal, there is a Newton polytope $P_{\text{gin}(I)}$ of $\mathbb{R}^m$ associated to it.  In this paper we introduce the \textit{generic initial system} $\{ \text{gin}(I^n) \}_n$ of $I$ and define the \textit{limiting polytope} of this system to be the the limit of the polytopes $\frac{1}{n}P_{\text{gin}(I^n)}$.  We study the case where $I$ is a complete intersection with minimal generators of degrees $d_1, d_2,  \dots, d_r$.  Our main result is the following theorem describing the limiting polytope of the reverse lexicographic generic initial system of such an ideal.

\begin{thm} 
\label{thm:limitingpolytope}
Let $I$ be a complete intersection of type $(d_1, \dots , d_r)$ in \linebreak[4] $K[x_1, \dots, x_r]$ where $d_1 \leq \cdots \leq d_r$  and $K$ is a field of characteristic 0.  The limiting polytope of the reverse lexicographic generic initial system $\{\text{gin}(I^n)\}_n$ is the closure of the complement in $\mathbb{R}_{\geq 0}^r$ of the $r$-simplex  with vertices at the points $$(0, 0, \dots, 0), (d_1, 0, \dots, 0), (0, d_2, 0, \dots, 0), \dots, (0, \dots, 0, d_r).$$
\end{thm}

This result easily extends to the case where $I$ is an $r$-complete intersection in $K[x_1, \dots, x_m]$  and shows that the asymptotic behavior of the generic initial system of a complete intersection is as nice as one could hope. The simplicity of the limiting polytope contrasts sharply with the structure of the individual ideals $\text{gin}(I^n)$ .

First, it is not clear that the ideals $\text{gin}(I^n)$ depend only on the degrees of the generators of $I$, even when $n=1$.  Consider the case where $R= k[x_1, x_2, x_3]$ and $I$ is a triple complete intersection with minimal generators of degrees $d_1, d_2$ and $d_3$.  Cimpoea\c{s} \cite{Cimpoeas06} has computed the generators of $\text{gin}(I)$ under the additional condition that $I$ be strongly Lefschetz; in this case $\text{gin}(I)$ is almost reverse lexicographic and depends only on $d_1, d_2,$ and $d_3$.  Cho and Park \cite{ChoPark07} prove that an ideal $J$ in $R=k[x_1, x_2, x_3]$ is strongly Lefschetz if and only if $\text{gin}(J)$ is almost reverse lexicographic.  Therefore, in this case, showing that $\text{gin}(I)$ depends only on the degrees of the generators of $I$ is equivalent to showing that all triple complete intersections are strongly Lefschetz; the latter question has been well-studied and seems difficult to prove in general (see \cite{MiglioreNagel11}).

Second, the generators of the ideals $\text{gin}(I^n)$ are complicated in even the simplest cases and they depend heavily on the relative magnitudes of $d_1, \dots, d_r$.  For explicit descriptions of the generators of $\text{gin}(I^n)$ see \cite{Mayes12} for the case where $r=2$ and \cite{Cimpoeas06} for the case where $n=1, r=3$, and $I$ is strongly Lefschetz.

The generic initial system of a complete intersection, then, follows the philosophy guiding the study of asymptotic objects: the ideals $\text{gin}(I^n)$ are complex but uniformity is gained in the limit.  We anticipate that this will hold for other generic initial systems; further research is required to determine conditions on the limiting polytopes of such systems.

\section*{Acknowledgements}

I would like to thank Karen Smith for introducing this problem to me, for guiding my work, and for her continued encouragement.  I would also like to thank Daniel Erman for suggestions on an earlier draft.  Finally,  I would like to thank Mel Hochster, Robert Lazarsfeld, Irena Swanson, and Michael Von Korff for useful discussions.  Calculations leading to the statement of the main theorem were performed using the computer software Macaulay2 \cite{Macaulay2}.
\section{Preliminaries}

In this section we will introduce some notation, definitions, and preliminary results related to generic initial ideals and systems of ideals. Throughout $R=K[x_1, \dots, x_m]$ is a polynomial ring over a field $K$ of characteristic 0 with the standard grading and some fixed term order $>$ with $x_1 > \cdots >x_m$.  A monomial $x_1^{j_1}x_2^{j_2} \cdots x_m^{j_m}$ of $R$ may also be written in multi-index notation as $x^J$ where $J=(j_1, \dots, j_m)$.  

%%Generic Initial Ideals 
\subsection{Generic Initial Ideals}

An element $g = (g_{ij}) \in \text{GL}_m(K)$ acts on $R$ and sends any homogeneous element $f(x_1, \dots, x_n)$ to the homogeneous element 
$$f(g(x_1), \dots, g(x_n))$$ 
where $g(x_i) = \sum_{j=1}^m g_{ij}x_j$.  If $g(I)=I$ for every upper triangular matrix $g$ then we say that $I$ is \textit{Borel-fixed}.  Borel-fixed ideals are \textit{strongly stable} when $K$ is of characteristic 0; that is, for every monomial $m$ in the ideal such that $x_i$ divides $m$, the monomials $\frac{x_jm}{x_i}$ for all $j<i$ are also in the ideal.  This property makes such ideals particularly nice to work with.

 %: given a Borel-fixed ideal $J$ it is relatively easy to read off the saturation degree, regularity, and minimal free resolution of $J$ and the dimension and depth of $R/J$.  

To any homogeneous ideal $I$ of $R$ we can associate a Borel-fixed monomial ideal $\text{gin}_{>}(I)$ which can be thought of as a coordinate-independent version of the initial ideal.  Its existence is guaranteed by Galligo's theorem (also see \cite[Theorem 1.27]{Green98}).

\begin{theoremb}[{Galligo's theorem \cite{Galligo74}}]
\label{thm:galligo}
For any multiplicative monomial order $>$ on $R$ and any homogeneous ideal $I\subset R$, there exists a Zariski open subset $U \subset \text{GL}_m$ such that $\text{In}_{>}(g(I))$ is constant and Borel-fixed for all $g \in U$.  
\end{theoremb}

\begin{defn}
The \textbf{generic initial ideal of $I$}, denoted $\text{gin}_{>}(I)$, is defined to be $\text{In}_{>}(g(I))$ where $g \in U$ is as in Galligo's theorem.
\end{defn}

The \textit{reverse lexicographic order} $>$ is a total ordering on the monomials of $R$ defined by 
\begin{enumerate}
\item if $|I| =|J|$ then $x^I > x^J$ if there is a $k$ such that $i_m = j_m$ for all $m>k$ and $i_k < j_k$; and
\item if $|I| > |J|$ then $x^I >x^J$.
\end{enumerate}
For example, $x_1^2 >x_1x_2 > x_2^2>x_1x_3>x_2x_3>x_3^2$.   From this point on, $\text{gin}(I) = \text{gin}_{>}(I)$ will denote the generic initial ideal with respect to the reverse lexicographic order.

The following theorem records two of the properties shared by $\text{gin}(I)$ and $I$.  The first statement is a consequence of the fact that Hilbert functions are invariant under making changes of coordinates and taking initial ideals.  The second statement is a result of Bayer and Stillman; for a simple proof see Corollary 2.8 of \cite{AhnMigliore07}.

\begin{thm}
\label{thm:commonproperties}
For any homogeneous ideal $I$ in $R$:
\begin{enumerate}
\item the Hilbert functions of $I$ and $\text{gin}(I)$ are equal; and
\item \text{depth}(R/I) = \text{depth}(R/\text{gin}(I)).
\end{enumerate}
\end{thm}

%%The generic initial system

\subsection{The Generic Initial System}

The focus of this paper is on understanding the behavior of the generic initial ideals of the powers of a fixed ideal.

\begin{defn}
The \textbf{generic initial system} of a homogeneous ideal $I$ is the collection of ideals $J_{\bullet}$ such that $J_i = \text{gin}(I^i)$.
\end{defn}

\begin{defn}[\cite{ELS01}]
A \textbf{graded system of ideals} is a collection of ideals $J_{\bullet} =\{J_i\}_{i=1}^{\infty}$ such that 
$$J_i \cdot J_j \subseteq J_{i+j} \hspace{0.3in} \text{ for all } i, j \geq 1.$$
\end{defn}

\begin{lem}
The generic initial system is a graded system of monomial ideals.
\end{lem}

\begin{proof}
By definition, $\text{gin}(I^i)$ is a monomial ideal. We need to show that for all $i,j \geq 1$, $\textnormal{gin}(I^i)\cdot \textnormal{gin}(I^j) \subseteq \textnormal{gin}(I^{i+j})$.  For any $l \geq 1$, let $U_l$ be the Zariski open subset of $GL_m$ such that $\text{gin}(I^l)= \text{In}(g \cdot (I^l))$ for all $g$ in $U_l$.  Since $U_i$, $U_j$, and $U_{i+j}$ are Zariski open they have a nonempty intersection; fix some $g \in U_i \cap U_j \cap U_{i+j}$.  
Given monomials $f' \in \textnormal{gin}(I^i) = \textnormal{In}(g(I^i))$ and $h' \in \textnormal{gin}(I^j) = \textnormal{In}(g(I^j))$, suppose that $f'=\textnormal{In}(g(f))$ and $h' = \textnormal{In}(g(h))$ for $f \in I^i$ and $h \in I^j$.  Now 
 $$f'\cdot h' = \textnormal{In}(g(f)) \textnormal{In}(g(h)) = \textnormal{In}(g(f) \cdot g(h)) = \textnormal{In}(g(f \cdot h)) \in \text{In}(g(I^{i+j}))$$
since $f \cdot h \in I^{i+j}$.  Thus $f'\cdot h' \in \text{gin}(I^{i+j})$ as desired.
\end{proof}

%%Volume and Multiplier Ideals
\subsection{Volume and multiplier ideals}
\label{sec:volmultideal}

In this section we will discuss the geometric interpretations of the volume and of the asymptotic multiplier ideal associated to a graded system of monomial ideals. 

\begin{defn} [\cite{ELS03}]
\label{defn:algebraicvolume}
Let $\mathrm{a}_{\bullet}$ be a graded system of zero-dimensional ideals in $R=K[x_1, \dots, x_m]$.  The \textbf{volume of $\mathrm{a}_{\bullet}$} is 
$$\mathrm{vol}(\mathrm{a}_{\bullet}) := \limsup_{n \rightarrow \infty} \frac{m! \cdot \mathrm{length}(R/\mathrm{a}_n)}{n^m}.$$
\end{defn}

Let $J$ be a monomial ideal of $R$.  We may regard $J$ as a subset $\Lambda$ of $\mathbb{N}^m$ consisting of the points $\lambda$ such that $x^{\lambda} \in J$.  The \textit{Newton polytope} $P_J$ of $J$ is the convex hull of $\Lambda$ regarded as a subset of $\mathbb{R}^m$.  Scaling the polytope $P_J$ by a factor of $r$ gives another polytope which we will denote $rP_J$.

If $\mathrm{a}_{\bullet}$ is a graded system of monomial ideals in $R$ the polytopes of $\{ \frac{1}{q} P_{\mathrm{a}_q} \}_q$ are nested: $\frac{1}{c}P_{\mathrm{a}_c} \subset \frac{1}{c+1}P_{\mathrm{a}_{c+1}}$ for all $c \geq 1$.  The \textit{limiting polytope $P$ of $\mathrm{a}_{\bullet}$} is the limit of the polytopes in this set:
$$P = \bigcup_{q \in \mathbb{N}^*} \frac{1}{q} P_{\mathrm{a}_q}.$$

Under the additional assumption that the ideals of $\mathrm{a}_{\bullet}$ are zero-dimensional, the closure of each set $\mathbb{R}^m_{\geq 0} \backslash P_{\mathrm{a}_q}$ in $\mathbb{R}^m$ is compact. This closure is denoted by $Q_q$ and we let 
$$Q = \bigcap_{q \in \mathbb{N}^*} \frac{1}{q} Q_q.$$
Note that finding the volume of $Q$, $\mathrm{vol}(Q)$, is the same as finding the volume underneath of the limiting polytope $P$.  It turns out that this geometric volume is closely tied to the algebraic volume of $\mathrm{a}_{\bullet}$.

\begin{prop} [\cite{Mustata02}]
\label{prop:volumesequal}
If $\textrm{a}_{\bullet}$ is a graded system of zero-dimensional monomial ideals in $R=K[x_1, \dots, x_m]$ and $Q$ is as defined above,
$$\mathrm{vol}(\textrm{a}_{\bullet}) = m! \mathrm{ vol}(Q).$$
\end{prop}

\begin{proof}
This is an immediate consequence of Theorem 1.7 and Lemma 2.13 of \cite{Mustata02}.
\end{proof}

Although multiplier ideals are an important tool in algebraic geometry, they are usually difficult to compute explicitly.  Since we will only be concerned with calculating multiplier ideals of monomial ideals we will give a definition only for this special case.  See \cite{Blickle04} for an introduction to multiplier ideals or \cite{Laz04} for a more general treatment. 

\begin{defn}  [\cite{Howald01}]
Let $J \subset R$ be a monomial ideal and let $P$ be its Newton polytope.  The \textbf{multiplier ideal of $J$ with coefficient $c$}, denoted $\mathcal{J}(c\cdot J)$, is the monomial ideal 
$$\mathcal{J}(c \cdot J) := \{x^{\lambda} : \lambda + \textnormal{\textbf{1}} \in \text{Int}(cP_J) \cap \mathbb{N}^m\}.$$
\end{defn}

The asymptotic definition requires the following lemma.

\begin{lem} [Lemma 1.3 of \cite{ELS01}]
\label{lem:asympjustification}
Let $J_{\bullet}$ be a graded system of ideals and fix a rational number $c>0$.  Then for $p \gg 0$ the multiplier ideals $\mathcal{J}(\frac{c}{p}\cdot J_p)$ coincide.
\end{lem}

\begin{defn}  
\label{defn:asympmultideal}
Let $J_{\bullet} = \{J_k\}_{k\in \mathbb{N}}$ be a graded system of ideals on $R$.  Given $c>0$ the \textbf{asymptotic multiplier ideal of $J_{\bullet}$ with coefficient $c$} is 
$$\mathcal{J}(c\cdot J_{\bullet}) = \mathcal{J}(\frac{c}{p} \cdot J_p)$$
for any sufficiently large $p$ (guaranteed by Lemma \ref{lem:asympjustification}).  When $J_{\bullet}$ is a graded system of monomial ideals with limiting polytope $P$,
$$\mathcal{J}(c \cdot J_{\bullet}) = \{x^{\lambda} : \lambda + \textnormal{\textbf{1}} \in \text{Int}(cP) \cap \mathbb{N}^m\}$$
for $p \gg 0$.
\end{defn}

Therefore, the volume and the asymptotic multiplier ideal of a graded system of monomial ideals are entirely determined by the limiting polytope $P$ of the system.
\section{The Generic Initial System of a Complete Intersection}

A homogeneous ideal $I = (f_1, \dots, f_r)$ is a \textit{complete intersection of type $(d_1, \dots, d_r)$} if $f_1, \dots, f_r$ is a regular sequence on $R$ and $\text{deg}(f_i) = d_i$.  Since the $f_i$ are homogeneous, any permutation of $f_1, \dots, f_r$ is still a regular sequence; therefore, we may assume that $d_1 \leq \cdots \leq d_r$.  

The goal of this section is to describe the reverse lexicographic generic initial ideals $\text{gin}(I^n)$ for such a complete intersection.  The following result tells us that no matter how many variables the ambient ring $R$ has the minimal generators of these generic initial ideals only involve the first $r$ variables.  

\begin{lem}
\label{lem:varsingenset}
Let $I$ be a type $(d_1, \dots, d_r)$ complete intersection in $R=K[x_1, \dots, x_m]$ (so that $m \geq r$) and let $A_n$ denote the set of minimal generators of $\text{gin}(I^n)$.  Then the monomials of $A_n$ are contained in $K[x_1,\dots, x_r]$ and the largest degree element of $A_n$ is of the form $x_r^{p_r(n)}$ for some $p_r(n) \geq 1$.
\end{lem}

This lemma is a consequence of the following result of Herzog and Srinivasan (see Lemma 3.1 of \cite{HS98}) which relates the depth and dimension of a Borel-fixed monomial ideal to the variables appearing in its minimal generating set.  

\begin{prop}
\label{prop:DMrelations}
Let $J$ be a Borel-fixed monomial ideal in $R$ and define 
$$D(J) := \text{max} \{t | x_t^j \in J \text{ for some positive integer } j \}$$
and 
$$M(J) := \text{max} \{t | x_t \text{ appears in some minimal generator of } J \}.$$
Then 
\begin{enumerate}
\item $\text{dim}(R/J) = m-D(J)$; and
\item $\text{depth}(R/J) = m-M(J)$.
\end{enumerate}
\end{prop}

Note that when $I$ is a complete intersection of type $(d_1, \dots, d_r)$ in $R$,
$$\text{dim}(R/I^n) = \text{depth}(R/I^n) = m-r$$
for all $n \geq 1$.
It then follows by Theorem \ref{thm:commonproperties} that the depth and dimension of $R/\text{gin}(I^n)$ are equal to $m-r$ as well.

\begin{proof}[Proof of Lemma \ref{lem:varsingenset}]
By Proposition \ref{prop:DMrelations},
$$D(\text{gin}(I^n)) = m-\text{dim}(R/\text{gin}(I^n)) = r = m- \text{depth}(R/\text{gin}(I^n)) = M(\text{gin}(I^n))$$
This means that the minimal generating set $A_n$ of $\text{gin}(I^n)$ is contained in $S=K[x_1, \dots, x_r]$ and that $A_n$ contains a power of $x_r$, say $x_r^{p_r(n)}$. 
The fact that $\text{gin}(I^n)$ is strongly stable means that we can replace each $x_r$ in $x_r^{p_r(n)}$ with any variable $x_1, \dots, x_r$ and still get an element of $\text{gin}(I^n)$.  Therefore, any monomial $x^J \in S$ of degree $p_r(n)$ is also contained in $\text{gin}(I^n)$.  Now it is clear that $A_n \subset S$ cannot contain any element of degree greater than $p_r(n)$.
\end{proof}

Given the fact that the generators of $\text{gin}(I^n)$ only involve the first $r$ variables it is natural to wonder whether we can restrict our attention to the case where $R=K[x_1, \dots, x_r]$ (that is, $r=m$).  The following result tells us that the answer is essentially `yes'.

\begin{prop}
\label{prop:independentofambient}
Given a type $(d_1, \dots, d_r)$ complete intersection $I$ in $R=K[x_1, \dots, x_m]$, so that $m \geq r$, there exists a type $(d_1, \dots, d_r)$ complete intersection $L$ of $S=K[x_1, \dots, x_r]$ such that the minimal generators of $\text{gin}(L^n)$ are the same as the minimal generators of $\text{gin}(I^n)$ for all $n \geq 1$.
\end{prop}

\begin{proof}
We will proceed by induction on $m$.  The statement is trivial in the case where $m=r$.  Assume that it holds for all $m \leq M $ such that $m \geq r$ and let $I$ be a complete intersection of type $(d_1, \dots, d_r)$ in $K[x_1, \dots, x_{M+1}]$.  If $h=a_1x_1 + \cdots + a_{M+1}x_{M+1}$ is a general linear form in $K[x_1, \dots, x_{M+1}]$, consider the ring isomorphism 
$$\phi: \frac{K[x_1, \dots, x_{M+1}]}{(h)} \rightarrow K[x_1, \dots, x_M]$$
given by sending $x_i$ to $x_i$ for $i=1, \dots, M$ and $x_{M+1}$ to $-\sum_{i=1}^{M} \frac{a_i}{a_{M+1}}x_i$.  If $J$ is an ideal of $K[x_1, \dots, x_M+1]$, let $J_H = \phi((J+(h))/(h))$.  Since $h$ is a general linear form, the ideal $I_H$ is a complete intersection of type $(d_1, \dots, d_r)$ in $K[x_1, \dots, x_M]$.

If 
$$\psi: K[x_1, \dots, x_{M+1}] \rightarrow K[x_1, \dots, x_M]$$
is the map given by sending $x_i$ to $x_i$ for $i=1, \dots, M$ and $x_M$ to $0$, we have the following well-known relation between the generic initial ideal of any ideal $J$ of $K[x_1, \dots, x_{M+1}]$ and the generic initial ideal of $J_H$ (see \cite{BS87} and Corollary 2.5 of \cite{Green98}):
$$\text{gin}(J_H) = \psi(\text{gin}(J)).$$
By Lemma 3.1, no minimal generator of $\text{gin}(I^n)$ involves the variable $x_{M+1}$ and so 
$$\text{gin}((I_H)^n) = \text{gin}((I^n)_H) = \psi(\text{gin}(I^n))$$
has the same generators as $\text{gin}(I^n)$.  
By the inductive assumption, there exists a type $(d_1, \dots, d_r)$ complete intersection $L \subset K[x_1, \dots, x_r]$ such that $\text{gin}((I_H)^n)$ - and thus $\text{gin}(I^n)$ - has the same minimal generators as $\text{gin}(L^n)$ for all $n \geq 1$.
Therefore, the statement holds for all $m \geq r$.
\end{proof}

%Consider a general linear form $h=a_1x_1 + \cdots + a_{M+1}x_{M+1}$ in $K[x_1, \dots, x_{M+1}]$ defining a hyperplane $H$ and let $I_H=(I+(h))/(h)$ be the ideal of the associated hyperplane section.  We may think of $I_H$ as an ideal of $K[x_1, \dots, x_M]$ obtained by replacing each $x_{M+1}$ in $I$ with $-\sum_{i=1}^{M} \frac{a_i}{a_{M+1}}x_i$.  Note that it does not matter whether we replace $x_{M+1}$ before or after taking a power so $(I^n)_H=(I_H)^n$.  We have the following well-known relation between the generic initial ideal of any ideal $J$ of $K[x_1, \dots, x_{M+1}]$ and the generic initial ideal of an generic hyperplane section $J_H$ (see \cite{BS87} and Corollary 2.5 of \cite{Green98}):
%$$\text{gin}(J_H) = \text{gin}(J)|_{x_{M+1} \rightarrow 0}.$$
%This means that $\text{gin}((I^n)_H)=\text{gin}((I_H)^n)$ has the same minimal generators as $\text{gin}(I^n)$ since by Lemma \ref{lem:varsingenset} no minimal generator of $\text{gin}(I^n)$ involves the variable $x_{M+1}$.  Since $h$ was chosen to be a \textit{generic} hyperplane section, $I_H$ is a complete intersection of type $(d_1, \dots, d_r)$ in $K[x_1, \dots, x_M]$.  By the inductive assumption, there exists a type $(d_1, \dots, d_r)$ complete intersection $L \subset K[x_1, \dots, x_r]$ such that $\text{gin}((I_H)^n)$ - and thus $\text{gin}(I^n)$ - has the same minimal generators as $\text{gin}(L^n)$ for all $n \geq 1$.
%Therefore, the statement holds for all $m \geq r$.
%\end{proof}

Fix an $r$-complete intersection $I$ of type $(d_1, \dots, d_r)$.  Let $p_i(n)$ denote the smallest power of $x_i$ contained inside of $\text{gin}(I^n)$ so that $(x_1^{p_1(n)}, x_2^{p_2(n)}, \dots, \linebreak[1] x_r^{p_r(n)})$ is the largest ideal generated by variable powers that is contained inside of $\text{gin}(I^n)$.

\begin{lem}
\label{lem:p1value}
$p_1(n) = nd_1$ for all $n \geq 1$.
\end{lem}

\begin{proof}
Since $\text{gin}(I^n)$ is strongly stable there is no element of $\text{gin}(I^n)$ of degree smaller than $p_1(n)$:  if there was such a monomial $x^J$ we could replace each $x_2, \dots, x_m$ in $x^J$ with $x_1$ and get a power of $x_1$ smaller than $p_1(n)$ contained in $\text{gin}(I^n)$. If $f_1$ is the generator of $I$ of degree $d_1$, the smallest degree element of $I^n$, and thus of $\text{gin}(I^n)$, is of degree $nd_1$.  Therefore, $p_1(n)=nd_1$.
\end{proof}

\begin{lem}
\label{lem:increasingps}
$p_1(n) \leq p_2(n) \leq \cdots \leq p_r(n)$ for all $n \geq 1$.
\end{lem}

\begin{proof}
Since $\text{gin}(I^n)$ is strongly stable we can replace each $x_i$ in $x_i^{p_i(n)} \in \text{gin}(I^n)$ with $x_{i-1}$ and stay inside of $\text{gin}(I^n)$: $x_{i-1}^{p_i(n)} \in \text{gin}(I^n)$. $x_{i-1}^{p_{i-1}(n)}$ is a minimal generator of $\text{gin}(I^n)$ so $p_{i-1}(n) \leq p_i(n)$ for each $i=2, \dots, r$.  
\end{proof}

To determine the value of $p_r(n)$ we will compare the \textit{Betti numbers} of $I^n$ and $\text{gin}(I^n)$. Let
$$0 \rightarrow F_m \rightarrow \cdots \rightarrow F_1 \rightarrow F_0 \rightarrow J \rightarrow 0$$
be the unique minimal free resolution of an ideal $J$.  The graded Betti numbers of $J$, $\beta_{i,j}(J)$, are defined by $F_i = \bigoplus_j R(-j)^{\beta_{i,j}(J)}$.
A \textit{consecutive cancellation} takes a sequence $\{\beta_{i,j}\}$ to a new sequence by replacing $\beta_{i,j}$ by $\beta_{i,j}-1$ and $\beta_{i+1,j}$ by $\beta_{i+1,j}-1$.  The `Cancellation Principle' says that the graded Betti numbers $\beta_{i,j}(I^n)$ of $I^n$ can be obtained from the graded Betti numbers of $\beta_{i,j}(\text{gin}(I^n))$ by making a series of consecutive cancellations (see Corollary 1.21 of \cite{Green98}).

%For example, if $I$ is a type (2,4) complete intersection we have the following minimal free resolutions of $I^2$ and $\text{gin}(I^2)$:
%$$ 0 \rightarrow R(-10)\oplus R(-8) \rightarrow R(-8) \oplus R(-6) \oplus R(-4) \rightarrow I^2 \rightarrow 0$$
%and 
%$$0 \rightarrow \mathcal{H}_1 \rightarrow \mathcal{H}_0 \rightarrow \text{gin}(I^2) \rightarrow 0$$
%where $\mathcal{H}_1 = R(-10) \oplus R(-9) \oplus R(-8) \oplus R(-7)$ and $\mathcal{H}_0 = R(-9) \oplus R(-8) \oplus R(-7) \oplus R(-6) \oplus R(-4)$.
%The copies of $R(-9)$ and $R(-7)$ in $\mathcal{H}_0$ and $\mathcal{H}_1$ cancel to give the terms of the minimal free resolution of $I^2$. Note that even though $R(-8)$ appears in both $\mathcal{H}_0$ and $\mathcal{H}_1$ it does not cancel; there is no way of determining which terms will cancel and which ones will not.

\begin{lem}
\label{lem:prvalue}
$p_r(n) = d_1 + \cdots + d_{r-1}+nd_r -r+1$ for all $n \geq 1$.
\end{lem}

\begin{proof}
Let $\text{gin}(I^n)$ have minimal generators $x^{J_1}, \dots, x^{J_N}$.  By a theorem of Eliahou and Kevaire (see \cite{EK90} or Theorem 1.31 of \cite{Green98}), the initial module of $p$th syzygies of $\text{gin}(I^n)$ is minimally generated by $x_{i_p} \otimes x_{i_{p-1}} \otimes \cdots \otimes x_{i_1} \otimes x^{J_j}$ where $1 \leq j \leq N$, $i_p < i_{p-1} < \cdots < i_1$, and $i_1$ is less than the index of the largest variable appearing in $J_j$.

Since $x_r^{p_r(n)}$ is a generator of $\text{gin}(I^n)$, $x_1\otimes x_2 \otimes \cdots \otimes x_{r-1}\otimes x_r^{p_r(n)}$ is in the initial module of $(r-1)$st syzygies and $\beta_{r-1,p_r(n)+(r-1)}(\text{gin}(I^n)) \neq 0$.  Further, since $\text{gin}(I^n)$ is Borel-fixed, the largest degree element in its minimal generating set is $x_r^{p_r(n)}$.  Thus, by Eliahou-Kervaire, the largest possible degree of a minimal generator of the initial module of $(r-2)$nd syzygies is $p_r(n)+(r-2)$; this means that $\beta_{r-2,p_r(n)+(r-1)}(\text{gin}(I^n))=0$.  Therefore, no consecutive cancellation can occur between $\beta_{r-2,p_r(n)+(r-1)}(\text{gin}(I^n))$ and $\beta_{r-1,p_r(n)+(r-1)}(\text{gin}(I^n))$ so, by the Cancellation Principle, $\beta_{r-1,p_r(n)+(r-1)}(I^n) \geq 1$ and $\beta_{r-1,s}(I^n) = 0$ for all $s > p_r(n)+(r-1)$.

By the main result of \cite{Guardo05}, $\beta_{r-1,d_1+ \cdots +d_{r-1}+nd_r}(I^n) \geq 1$ and $\beta_{r-1,s}(I^n) = 0$ for all $s > d_1+ \cdots +d_{r-1}+nd_r$.  Therefore, 
$$d_1+ \cdots+ d_{r-1} +nd_r = p_r(n)+r-1$$
and $p_r(n) = d_1+d_2+\cdots + d_{r-1}+nd_r-r+1$ as claimed.

%Let
%$$0 \rightarrow \mathcal{H}_{r-1} \rightarrow \cdots \rightarrow \mathcal{H}_1 \rightarrow \mathcal{H}_0\rightarrow \text{gin}(I^n) \rightarrow 0$$
%be the minimal free resolution of $\text{gin}(I^n)$.   Since $x_r^{p_r(n)}$ is a generator of $\text{gin}(I^n)$, $x_1\otimes x_2 \otimes \cdots \otimes x_{r-1}\otimes x_r^{p_r(n)}$ is in the initial module of $(r-1)$st syzygies and $R(-p_r(n)-(r-1))$ is a summand of $\mathcal{H}_{r-1}$. There is no generator of degree greater than $p_r(n)$ by Lemma \ref{lem:varsingenset} and so there is no summand of $\mathcal{H}_{r-2}$ of degree $(-p_r(n)-(r-1))$.  Therefore, $R(-p_r(n)-(r-1))$ in $\mathcal{H}_{r-1}$ cannot cancel with any adjacent term.  If
%$$0 \rightarrow \mathcal{G}_{r-1} \rightarrow \cdots \rightarrow \mathcal{G}_1 \rightarrow \mathcal{G}_0\rightarrow I^n \rightarrow 0$$
%is the minimal free resolution of $I^n$, the Cancellation Principle tells us that the summand of $\mathcal{G}_{r-1}$ with the greatest shift is $R(-p_r(n)-(r-1))$.   By the main result of \cite{Guardo05} this summand is $R(-d_1-d_2- \cdots -d_{r-1}-nd_r)$.  Therefore,
%$$p_r(n) + r - 1 = d_1+d_2+\cdots+d_{r-1}+nd_r$$
%so $p_r(n) = d_1+d_2+\cdots + d_{r-1}+nd_r-r+1$ as claimed.
\end{proof}

%Let $\mathcal{M}_{(r,s,t)}$ denote the set of $(a_1, \dots, a_r) \in \mathbb{N}^r$ such that $a_1 + \cdots a_r = s$and at least $t$ of the $a_i$'s are nonzero.  
 %The minimal free resolution of $I^n$ has the form
%$$0 \rightarrow \mathcal{H}_{r-1} \rightarrow \mathcal{H}_{r-2} \rightarrow \cdots \rightarrow \mathcal{H}_0 \rightarrow I^n \rightarrow 0$$
%where 
%$$\mathcal{H}_0 = \bigoplus_{(a_1, \dots, a_r) \in \mathcal{M}_{r,n+i, 1}} R(-a_1d_1- \cdots -a_rd_r)$$
%and for $i=1, \dots, r-1$, 
%$$\mathcal{H}_i = \bigoplus_{l_1=i+1}^r \bigg[ \bigoplus_{l_2=l_1}^r \bigg[ \cdots \bigg[ \bigoplus_{l_i=l_{i-1}}^r \bigg[ \bigoplus_{(a_1, \dots, a_r) \in \mathcal{M}_{r,n+i, l_i}} R(-a_1d_1- \cdots - a_rd_r) \bigg] \bigg] \cdots \bigg] \bigg]$$

\section{The Limiting Polytope of the Generic Initial System of a Complete Intersection}

In this section we prove Theorem \ref{thm:limitingpolytope} describing the asymptotic behavior of the reverse lexicographic generic initial system of a complete intersection.  Throughout $I$ will be a complete intersection of type $(d_1, \dots, d_r)$ and $p_i(n)$ will be the minimal power of $x_i$ contained in $\text{gin}(I^n)$ as described in the previous section.  The following result bounding the growth of the $p_i(n)$ will be used in the proof of the main theorem.

\begin{lem}
\label{lem:limitofpi}
Suppose that $I$ is a type $(d_1, \dots, d_r)$ complete intersection in $S=K[x_1, \dots, x_r]$.  For all $i \leq r$, $$\lim_{n \rightarrow \infty} \frac{p_i(n)}{n} \leq d_i.$$
\end{lem}

\begin{proof}
We will proceed by induction on $r$.  If $r=2$ Lemma \ref{lem:p1value} gives 
$$\lim_{n\rightarrow\infty} \frac{p_1(n)}{n} = \lim_{n\rightarrow\infty}  \frac{nd_1}{n} = d_1$$ 
and Lemma \ref{lem:prvalue} gives 
$$\lim_{n\rightarrow\infty} \frac{p_2(n)}{n} = \lim_{n\rightarrow\infty}  \frac{d_1+nd_2-1}{n} = d_2$$ so the result holds in this case. 

 Assume that  the statement holds for all $r \leq R$ and that $I$ is a complete intersection of type $(d_1, \dots, d_{R+1})$ in $K[x_1, \dots, x_{R+1}]$ generated by homogeneous polynomials $f_1, \dots, f_{R+1}$ where $\text{deg}(f_i) = d_i$.  Consider the ideal $J \subseteq K[x_1, \dots, x_{R+1}]$ generated by $f_1, \dots, f_R$ and note that $J \subseteq I$ implies $\text{gin}(J^{n}) \subseteq  \text{gin}(I^n)$.  If $p'_i(n)$ denotes the minimal power of $x_i$ contained in $\text{gin}(J^n)$, $p'_i(n) \geq p_i(n)$ for all $i=1, \dots, R$.  By Proposition \ref{prop:independentofambient} there exists an ideal $L \subset R[x_1, \dots, x_R]$ such that the minimal generators of $\text{gin}(L^n)$ are the same as the minimal generators of $\text{gin}(J^n)$ for all $n\geq 1$; in particular, the minimal power of $x_i$ contained in $\text{gin}(L^n)$ is equal to the minimal power $p'_i(n)$ contained in $\text{gin}(J^n)$.  The statement of the lemma holds for $L$ by the inductive assumption so
$$\lim_{n \rightarrow \infty} \frac{p_i(n)}{n} \leq \lim_{n \rightarrow \infty} \frac{p'_i(n)}{n} \leq d_i$$
for all $i \leq R$.  For $i=R+1$ we have 
$$\lim_{n \rightarrow \infty} \frac{p_{R+1}(n)}{n} = \lim_{n \rightarrow \infty} \frac{d_1 + \cdots + d_{R}+nd_{R+1}-(R+1)+1}{n} = d_{R+1}$$
by Lemma \ref{lem:prvalue}.  Therefore, the claim holds for $r=R+1$ and thus for all $r$.
\end{proof}

Suppose that $T$ is an $r$-simplex in $\mathbb{R}^r$ with vertices at the origin and at $\vec{v_1}, \vec{v_2}, \dots, \vec{v_r}$.  Recall that the volume of $T$ is equal to 
$$\mathrm{vol}(T) = \frac{\text{det}(A)}{r!}$$
where $A$ is a matrix with columns $\vec{v_1},  \dots, \vec{v_r}$ arranged so that $A$ has a positive determinant (see, for example, Chapter 5 of \cite{Lax96}).

We are now prepared to prove that the limiting polytope of the generic initial system $\mathrm{a}_{\bullet}$ of a type $(d_1, \dots, d_r)$ complete intersection in $K[x_1, \dots, x_r]$ is the closure of the complement in $\mathbb{R}^r_{\geq 0}$ of the $r$-simplex with vertices at the points $$(0, 0, \dots, 0), (d_1, 0, \dots, 0), (0, d_2, 0, \dots, 0), \dots, (0, \dots, 0, d_r).$$  As in Section \ref{sec:volmultideal}, the limiting polytope is denoted by $P$ and the closure of its complement in $\mathbb{R}^r_{\geq 0}$ is denoted by $Q$.

\begin{proof}[Proof of Theorem \ref{thm:limitingpolytope}]
First note that, since $x_i^{p_i(n)} \in \text{gin}(I^n)$ for $i=1, \dots, r$, there is a vertex of the limiting polytope lying on each coordinate axis of $\mathbb{R}^r$.  The vertex on the $x_i$ axis has nonzero coordinate $\lim_{n \rightarrow \infty} \frac{p_i(n)}{n}$ which is at most $d_i$ by Lemma \ref{lem:limitofpi}.  Under these conditions, the maximum possible volume beneath the convex limiting polytope $P$ is attained if and only if its boundary is defined by the coordinate planes and the hyperplane through the points $(d_1, 0, \dots, 0), \allowbreak (0, d_2, 0, \dots, 0), \allowbreak \dots, (0, \dots, 0, d_r)$.  In this case, $Q$ is the $r$-simplex described in the theorem and  
$$\mathrm{vol}(Q) = \frac{d_1\cdot d_2 \cdots d_r}{r!}.$$
To show that the maximum volume is attained and thus that the limiting polytope is as claimed we will compute the algebraic  volume of the graded system.  It follows from Exercise 12.3 of \cite{Eisenbud04} and the fact that $\text{length}(I^n/I^{n+1}) = \text{length}(R/I^{n+1})-\text{length}(R/I^n)$ that $\text{length}(R/I^n) = {n+r-1\choose r}d_1\cdots d_r$.  Since the lengths of $R/I^n$ and $R/\text{gin}(I^n)$ are equal, 
the volume of the generic initial system $\textrm{a}_{\bullet}$ of $I$ is 
\begin{eqnarray*}
\mathrm{vol}(\mathrm{a}_{\bullet}) &=&\limsup_{n\rightarrow \infty} \frac{r! \cdot \text{length}(R/\text{gin}(I^n))}{n^r} \\
&=& \limsup_{n \rightarrow \infty} r! \cdot {n+r-1\choose r}\frac{d_1 \cdots d_r}{n^r}\\
&=& \limsup_{n \rightarrow \infty} r! \cdot \frac{(n+r-1)(n+r-2) \cdots (n)}{r!} \frac{d_1 \cdots d_r}{n^r} \\
&=&  d_1\cdots d_r
\end{eqnarray*}
By Proposition \ref{prop:volumesequal}, $\mathrm{vol}(\mathrm{a}_{\bullet}) = \mathrm{vol}(Q)r!$ so the maximum possible geometric volume is attained.
\end{proof}

By Proposition \ref{prop:independentofambient}, if $I$ is a type $(d_1, \dots, d_r)$ complete intersection in any polynomial ring $K[x_1, \dots, x_m]$, there exists a complete intersection $L \subset K[x_1, \dots, x_r]$ of the same type such that the minimal generators of $\text{gin}(L^n)$ are the same as the minimal generators of $\text{gin}(I^n)$.  Thus, the above theorem gives the limiting polytope $P$ of \textit{any} complete intersection.  In particular, $P$ consists of the points $(\lambda_1, \dots, \lambda_m) \in \mathbb{R}^m_{\geq 0}$ such that
$$1 \leq \frac{\lambda_1}{d_1}+ \cdots + \frac{\lambda_r}{d_r}.$$
Note that there are no conditions on the last $m-r$ coordinates.
 
From Section \ref{sec:volmultideal}, knowing the limiting polytope of a graded system is enough to compute its asymptotic multiplier ideal. 

\begin{cor}
\label{cor:asympmultideal}
Let $I$ be a complete intersection of type $(d_1, \dots, d_r)$ in $K[x_1, \dots, x_m]$ and let $\mathrm{a}_{\bullet}$ denote the generic initial system of $I$; that is, $\mathrm{a}_n = \text{gin}(I^n)$.  Then the asymptotic multiplier ideal of $\mathrm{a}_{\bullet}$ is
$$
\mathcal{J}(\mathrm{a}_{\bullet}) = \{x_1^{c_1}x_2^{c_2} \cdots x_m^{c_m} | (c_1, \dots, c_m) \in \mathbb{N}^m  \textrm{ and }  \frac{(c_1+1)}{d_1} +  \cdots + \frac{(c_r+1)}{d_R} > 1\}
$$
\end{cor}

\begin{proof}
By the discussion above, the hyperplane $1 = \frac{x_1}{d_1} + \cdots + \frac{x_r}{d_r}$ bounds the limiting polytope of the generic initial system so, using Definition \ref{defn:asympmultideal}, $\mathcal{J}(\mathrm{a}_{\bullet})$ is as claimed. 
\end{proof}

%\appendix
%\include{AppendixExamples}
%\include{AppendixFormulas}
%\include{AppendixCalculationDetails}

\bibliography{Gin}
\bibliographystyle{amsalpha}
\nocite{*}

\end{document}